\documentclass[11pt, reqno]{amsart}
\usepackage{amssymb, amstext, amscd, amsmath}
\usepackage[all,cmtip]{xy}
\xyoption{all}
\usepackage{color}
\usepackage{hyperref} 

\theoremstyle{plain}
\newtheorem{theorem}{Theorem}[section]
\newtheorem{thm}[theorem]{Theorem}
\newtheorem{cor}[theorem]{Corollary}
\newtheorem{prop}[theorem]{Proposition}
\newtheorem{lem}[theorem]{Lemma}
\theoremstyle{definition}

\newtheorem{example}[theorem]{Example}

\newtheorem{defn}[theorem]{Definition}

\newcommand{\bC}{{\mathbb{C}}}

\newcommand{\bF}{{\mathbb{F}}}

\newcommand{\bN}{{\mathbb{N}}}

\newcommand{\bT}{{\mathbb{T}}}

  \newcommand{\A}{{\mathcal{A}}}
  \newcommand{\B}{{\mathcal{B}}}

\renewcommand{\H}{{\mathcal{H}}}
  
  \newcommand{\J}{{\mathcal{J}}}
  
\renewcommand{\L}{{\mathcal{L}}}
\newcommand{\M}{{\mathcal{M}}}
  \newcommand{\N}{{\mathcal{N}}}
\renewcommand{\O}{{\mathcal{O}}}
\renewcommand{\P}{{\mathcal{P}}}

\renewcommand{\S}{{\mathcal{S}}}
  
  \newcommand{\U}{{\mathcal{U}}}
  \newcommand{\V}{{\mathcal{V}}}
  \newcommand{\W}{{\mathcal{W}}}

\newcommand{\fA}{{\mathfrak{A}}}

\newcommand{\fI}{{\mathfrak{I}}}
\newcommand{\fJ}{{\mathfrak{J}}}

\newcommand{\fM}{{\mathfrak{M}}}

\newcommand{\fS}{{\mathfrak{S}}}
\newcommand{\fT}{{\mathfrak{T}}}

\renewcommand{\phi}{\varphi}
\newcommand{\upchi}{{\raise.35ex\hbox{\ensuremath{\chi}}}}

\newcommand{\AND}{\text{ and }}

\newcommand{\alg}{\operatorname{alg}}

\newcommand{\id}{{\operatorname{id}}}

\newcommand{\spn}{\operatorname{span}}

\newcommand{\ca}{\mathrm{C}^*}

\newcommand{\Fn}{\mathbb{F}_n^+}

\newcommand{\Fth}{\mathbb{F}_\theta^+}

\newcommand{\<}{\langle}
\renewcommand{\>}{\rangle}

\newcommand{\mt}{\varnothing}

\newcommand{\ol}{\overline}

\newcommand{\wot}{\textsc{wot}}

\newcommand{\gr}{\bF_\theta^+}

\begin{document}
\title
{Nonself-adjoint $2$-graph algebras}

\author[A. H. Fuller]{Adam H. Fuller}
\address{Adam H. Fuller, Department of Mathematics, University of Nebraska-Lincoln, Lincoln, NE, 68588-0130, USA}
\email{afuller7@math.unl.edu}

\author[D. Yang]
{Dilian Yang}
\address{Dilian Yang,
Department of Mathematics $\&$ Statistics, University of Windsor, Windsor, ON
N9B 3P4, CANADA} \email{dyang@uwindsor.ca}

\begin{abstract}
We study the structure of weakly-closed nonself-adjoint algebras arising from representations of single vertex $2$-graphs. These are the algebras generated by $2$ isometric tuples which satisfy a certain commutation relation. We show that these algebras have a lower-triangular $3\times 3$ form. The left-hand side of this matrix decomposition is a slice of the enveloping von Neumann algebra generated by the $2$-graph algebra. We further give necessary and sufficient conditions for these algebras themselves to be von Neumann algebras. The paper concludes with further study of atomic representations.
\end{abstract}

\thanks{The second author partially supported by an NSERC Discovery grant.}

\date{}
\maketitle

\section{Introduction}
\noindent Higher-rank graph $C^*$-algebras have been the subject of much research since their introduction by Kumjian and Pask \cite{KumPas}. These algebras serve as a higher-rank version of graph $C^*$-algebras. Their theory has been developed by Kumjian, Pask, Raeburn, Sims, to name but a few. See \cite{FMY, KPS, PRRS, RSY} and references therein for more on higher-rank graph $C^*$-algebras.

The nonself-adjoint counter-parts of higher-rank graph $C^*$-algebras were initially studied by Kribs and Power \cite{KriPow06}. Kribs and Power had previously studied the nonself-adjoint analytic algebras arising from the left-regular representation of a directed graph \cite{KriPow04}. This work serves as a natural generalisation of the study of the noncommutative Toeplitz algebra $\L_n$, initiated by Davidson and Pitts \cite{DP}. From this point of view, the nonself-adjoint $k$-graph algebras studied by Kribs and Power are a further generalisation of the noncommutative Toeplitz algebra $\L_n$. The study of the left-regular representation of a $k$-graph, with special attention payed to the case of single vertex $k$-graphs, was further developed by Power \cite{Power}. The representation theory of single vertex $k$-graphs was developed in a series of papers by Davidson, Power and the second author \cite{DPYatom, DPYdiln, DY1, DY2}. The nonself-adjoint algebras arising from finitely correlated representations of $k$-graphs have been dealt with by the first author as a special case of product systems of $C^*$-correspondences \cite{Ful}.
The classification of von Neumann algebras associated to single vertex 2-graphs was studied by the second author 
in \cite{Yang1, Yang2}.

Free semigroup algebras are the unital \textsc{wot}-closed algebras generated by row-isometries. The noncommutative Toeplitz algebra is an example of a free semigroup algebra. A row-isometry is an isometric operator from the Hilbert space $\H\oplus\cdots\oplus\H$ ($n$ times) to $\H$. A row-isometry is thus determined by $n$ isometries $S_1,\ldots,S_n$ on $\H$ with pairwise orthogonal ranges. They are the natural higher-dimensional generalisation of isometries and appear throughout mathematics and mathematical physics. Free semigroup algebras are the algebras which best encapsulate the representation theory of row-isometries (see \cite{Ken2}). Isometric representations of single vertex $k$-graph algebras are determined by $k$ row-isometries 
$[S^{(i)}_1,\ldots,S^{(i)}_{n_i}]$ ($1\leq i\leq k$) satisfying certain commutation rules. Thus, the isometric representation theory of single vertex $k$-graphs is a higher-dimensional analogue of the study of isometries.

In this paper we study the nonself-adjoint algebras arising from single vertex $2$-graphs further. Mirroring the case of a single row-isometry, i.e. the free semigroup algebra case,  we pay particular attention to the \textsc{wot}-closed algebras. In Section 2 relevant background definitions and results in $2$-graph algebras and free semigroup algebras are discussed.

In Section \ref{sec: analytic} we study when the norm-closed and \textsc{wot}-closed algebras arising from our representations are, in some sense, comparable to those arising from the left-regular representation. A result of Popescu \cite{Pop} says that the norm-closed unital algebra generated by a single row-isometry is completely isometrically isomorphic to the norm-closed unital algebra generated by the left-regular representation of $\bF_n^+$, the noncommutative disc algebra. The analogous result does not hold in general for isometric representations of $2$-graphs. We say that a representation is rigid when this property holds. In Theorem \ref{thm: rigidity} we give sufficient conditions for rigidity.

One of these conditions is on the $2$-graph itself, not on a representation. This is the condition of aperiodicity. Aperiodicity in graphs was studied in detail by Davidson and the second author \cite{DY1}. Let $S$ and $T$ be two row-isometries determining a representation of a $2$-graph. From the perspective of Theorem \ref{thm: rigidity}, assuming that the $2$-graph is aperiodic ensures some independence between $S$ and $T$. In particular we can not have $S=T$.

If $\gr$ is a $2$-graph, denote by $\L_\theta$ the unital \textsc{wot}-closed algebra generated by the left-regular representation of $\gr$. Let $\fS$ be the unital \textsc{wot}-closed algebra generated by an isometric representation of $\gr$ determined by row-isometries $S$ and $T$. We say that $\fS$ is analytic if the canonical map from $\fS$ to $\L_\theta$ sending generators to generators is a weak$^*$-weak$^*$ homeomorphism. Assuming that $\fS$ is generated by a rigid representation allows us to give a more practical description of when $\fS$ is analytic. For example, in Lemma \ref{lem: wandering} we can show that if $\fS$ is generated by a rigid representation and acts on a space spanned by its wandering vectors, then $\fS$ is analytic.

The main result of section 4 is Theorem \ref{thm: struct thm} in which we show the \textsc{wot}-closed algebra generated by a representation of a single vertex $2$-graph has a 
lower-triangular $2\times 2$ matrix form. Let $\fS$ be such an algebra, and let $\fM$ be  the von Neumann algebra generated by $\fS$. We show that there is a projection $P$ in $\fS$ such that $\fS P=\fM P$ and $P^{\perp}\H$ is invariant under $\fS$. Thus the projection $P$ determines the lower-triangular form for $\fS$. This extends results of the first author \cite{Ful} for the case when $P$ projects onto a finite dimensional space. A similar result for free semigroup algebras, the Structure Theorem for free semigroup algebras, was proved by Davidson, Katsoulis and Pitts \cite{DKP}. Our proof relies on the fact that there are many free semigroup algebras sitting inside $\fS$. The lower-triangular form of these free semigroup algebras forces the algebra $\fS$ to have a similar decomposition. We call the projection $P$ the first structure projection for the nonself-adjoint $2$-graph algebra $\fS$. We study $P$ further in section \ref{sec: 1st proj}. 

In general, it is not known how the $(2,2)$-entry of the lower-triangular form established for a nonself-adjoint $2$-graph algebra $\fS$ behaves. In section \ref{sec: 2nd proj} we show that the largest invariant subspace on which $\fS$ is analytic lies in the range of the $(2,2)$-entry. This determines a lower-triangular $3\times 3$ decomposition of the algebra with the $(3,3)$-entry being analytic.

In section 5 we consider the case when the \textsc{wot}-closed nonself-adjoint algebra $\fS$ generated by a representation of a $2$-graph is a von Neumann algebra. Several sufficient conditions for this to happen are given. Examples of such algebras rely on starting with a free semigroup algebra which is self-adjoint. The only known example of this is due to Read \cite{Read}.

In the final section we consider the special case of atomic representations. Atomic representations were classified by Davidson, Power and the second author \cite{DPYatom}, where they show that atomic representations break into $3$ different classes, each with various subclasses. In this section we give sufficient conditions for these representations to have wandering vectors. The existence of wandering vectors guarantees that the $(3,3)$-entry in the lower-triangular form of the nonself-adjoint $2$-graph algebra, as established in section \ref{sec: 2nd proj}, is non-zero. Here again, aperiodicity will play a role. We show that there are atomic representations of periodic $2$-graphs with no wandering vectors.

\section{Preliminaries and Notation}
\subsection{\texorpdfstring{Nonself-adjoint $2$-graph algebras}{Nonself-adjoint 2-graph algebras}}
\noindent Higher-rank graph algebras were introduced in 2000 by Kumjian and Pask \cite{KumPas}. Their definition relies on small categories. In this paper we are only concerned with single vertex $2$-graphs so we are afforded a simpler definition.

\begin{defn}\label{defn: 2-graph}
	Let $m$ and $n$ be positive integers and $\theta$ be a permutation in $S_{m\times n}$. We define
	the \emph{single vertex $2$-graph} $\gr$ to be the cancellative semigroup $\gr$ generated by the
	$m+n$ elements $e_1,\ldots,e_m$ and $f_1,\ldots,f_n$ satisfying
	\begin{enumerate}
		\item $e_1,\ldots,e_m$ form a copy of the free semigroup $\bF_m^+$,
		\item $f_1,\ldots,f_n$ form a copy of the free semigroup $\bF_n^+$,
		\item $e_if_j=f_{j'}e_{i'}$ when $\theta(i,j)=(i',j')$.
	\end{enumerate} 
\end{defn}
For a word $u=i_1i_2\ldots i_k\in\bF_m^+$ we write $e_u$ in place of $e_{i_1}e_{i_2}\ldots e_{i_k}$. Similarly we write $f_v$ for words $v\in\bF_n^+$. Take any $w\in\gr$. The commutation rules on $\gr$ mean that $w$ can be written uniquely as $w=e_{u}f_{v}$ for some $u\in\bF_m^+$ and $v\in\bF_n^+$. We define the \emph{degree} of $w$, written $d(w)$, by $d(w)=(|u|,|v|)$ where $|u|$ is the length of $u$ and $|v|$ is the length of $v$. We define the \emph{length} of $w$, written $|w|$, by $|w|=|u|+|v|$.

The above definition of a single vertex $2$-graph may seem far removed from what initially springs to mind on hearing the word ``graph''. We can describe $\gr$ alternatively in a way that, though not as convenient for computations, betrays the origins of the definition. Consider a single vertex $v$ with $m$ blue directed edges from $v$ to $v$ and $n$ red directed edges from $v$ to $v$. Label the blue edges $e_1$ to $e_m$ and label the red edges $f_1$ to $f_n$. Now we equate red-blue edges with blue-red edges using $\theta$ to determine the pairings. That is, if we travel by red path $f_j$ and then blue path $e_i$ we will consider this path the same as traveling by $e_{i'}$ and then $f_{j'}$, when $\theta(i,j)=(i',j')$. Then $\gr$ is simply the path space of this $2$ coloured graph. This description is useful to keep in mind, but we will rely more heavily on the description in Definition \ref{defn: 2-graph}.

The concept of aperiodicity for higher-rank graphs was introduced by Kumjian and Pask \cite{KumPas}. For single vertex $2$-graphs this idea was explored further by Davidson and the second author \cite{DY1}. Periodicity refers to, essentially, a necessary repetition in infinite paths of alternating $e$'s and $f$'s, i.e. an infinite path of alternating blue and red edges. We note the following characterisation found in \cite[Theorem 3.1]{DY1}, and refer the reader to \cite{DY1} for further information.

\begin{thm}\label{thm: periodic cond}
	If $2\leq m,n$ then $\gr$ is \emph{periodic} with period $(a,-b)$ if and only if there is a bijection
	\[
		\gamma:\{u\in\bF_m^+: |u|=a\}\rightarrow\{v\in\bF_n^+: |v|=b\}
	\]
	such that 
	\[
		e_uf_v=f_{\gamma(u)}e_{\gamma^{-1}(v)}.
	\]

	If $m=1$ or $n=1$ then $\gr$ is periodic.
\end{thm}
While many of our results will hold for both periodic and aperiodic $2$-graphs, we will frequently assume that we are dealing with aperiodic $2$-graphs in order to obtain stronger results.

\begin{defn}
Let $\gr$ be a $2$-graph. Let $S_1,\ldots,S_m$ and $T_1,\ldots,T_n$ be two sets of operators on a Hilbert space $\H$ which satisfy
	\[
		S_iT_j=T_{j'}S_{i'}
	\]
when $\theta(i,j)=(i',j')$. Let $S=[S_1,\ldots,S_m]$ be the row-operator from $\H^{(m)}$ to $\H$ and let
$T=[T_1,\ldots,T_n]$ be the row-operator from $\H^{(n)}$ to $\H$. We say that the pair $(S,T)$ is
\begin{enumerate}
	\item a \emph{contractive representation} of $\gr$ if both $S$ and $T$ are contractions,
	\item an \emph{isometric representation} of $\gr$ if both $S$ and $T$ are isometries,
	\item of \emph{Cuntz-type} if both $S$ and $T$ are unitaries.
\end{enumerate}
Equivalently, the pair $(S,T)$ is a contractive representation if
	\[
		\sum_{i=1}^mS_iS_i^*\leq 1\AND\sum_{j=1}^nT_jT_j^*\leq 1.
	\]
The pair $(S,T)$ is an isometric representation precisely when both $S$ and $T$ are isometric tuples, i.e. the $S_i$ and $T_j$ satisfy the Cuntz relations
	\[
		S_i^*S_j=\delta_{i,j}I \AND T_k^*T_l=\delta_{k,l}I.
	\]
An isometric representation $(S,T)$ is of Cuntz-type when both $S$ and $T$ are defect-free, i.e.
	\[
		\sum_{i=1}^mS_iS_i^*=I=\sum_{j=1}^nT_jT_j^*.
	\]
For $w=e_uf_v\in\gr$ we will write $(ST)_w$ for the operator $S_{u}T_{v}$.
\end{defn}

\begin{defn}
	Let $(S,T)$ be an isometric representation of a $2$-graph $\gr$. The
	\emph{nonself-adjoint $2$-graph algebra} $\fS$ generated by $(S,T)$ is the unital, weakly-closed algebra
	generated by the tuples $S$ and $T$. That is,
	\[
		\fS=\overline{\alg}^{\textsc{wot}}\{I,S_1,\ldots,S_m,T_1,\ldots,T_n\}.
	\]
\end{defn}

We are concerned in this paper with noncommutative operator algebras. Thus we will assume throughout this work that for a $2$-graph $\gr$ with $\theta\in S_{m\times n}$ 
either $m>1$ or $n>1$. Note that if  $m=n=1$ and $(S,T)$ is an isometric representation of $\gr$ then $S$ and $T$ are commuting isometries (not row-isometries) and so the nonself-adjoint $2$-graph algebra they generate is commutative.

While we are primarily interested in nonself-adjoint $2$-graph algebras generated by representations of Cuntz-type,
the following example, which is not of Cuntz-type, is motivating.

\begin{example}
	Let $\gr$ be a $2$-graph where $\theta\in S_{m\times n}$.
	Let $\H_\theta=\ell^2(\gr)$ be the separable Hilbert space with orthonormal basis $\{\xi_w: w\in\gr\}$.
	Define operators $E_i$ and $F_j$ for $1\leq i\leq m$, $1\leq j\leq n$ by
	\[
		E_i\xi_w=\xi_{e_iw}\AND F_j\xi_w=\xi_{f_jw}\quad (w\in\Fth).
	\]
	Let $E=[E_1,\ldots,E_m]$ and $F=[F_1,\ldots,F_n]$. Then $(E,F)$ is an isometric representation of
	$\gr$. This is called the \emph{left-regular representation} of $\gr$. We denote by $\L_\theta$ the
	nonself-adjoint $2$-graph algebra generated by $(E,F)$. We further denote by $\A_\theta$ the 
	norm-closed unital algebra generated by the representation $(E,F)$.
\end{example}

\subsection{Free semigroup algebras}
Nonself-adjoint $2$-graph algebras are the natural $2$ variable generalisation of free semigroup algebras. They also contain free semigroup algebras which are intrinsic to their structure. We will thus review some free semigroup algebra theory here. For the current state of the art on free semigroup algebras we refer the reader to the survey article \cite{Dav01} and to the more recent articles \cite{Ken, Ken2}. Throughout this article we will in particular rely on the structure of free semigroup algebras as described in \cite{DKP}.

\begin{defn}
	A \textit{free semigroup algebra} is a unital, weakly-closed algebra generated by a row-isometry. That is,
	$\fS$ is a free semigroup algebra if 
	\[
		\fT=\overline{\alg}^\textsc{wot}\{I,T_1,\ldots,T_n\},
	\]
	where $T_1,\ldots,T_n$ are isometries satisfying the Cuntz relations $T_i^*T_j=\delta_{i,j}I$.
\end{defn}
The name \emph{free semigroup algebra} arises from the fact that any row-isometry defines a representation of a free semigroup.

As in the case of nonself-adjoint $2$-graph algebras, the left-regular representation is an important example in the theory of free semigroup algebras.

\begin{example}[Left regular representation]\label{eg: fsg left-reg}
	Let $\H_n=\ell^2(\bF_n^+)$ be the Hilbert space with orthonormal basis $\{\xi_w: w\in\bF_n^+\}$. We let
	$L_1,\ldots,L_n$ be the left-regular representation operators on $\H_n$, i.e.
	\[
		L_i\xi_w=\xi_{iw}\quad (1\le i \le n,\,  w\in \Fn).
	\]
	We denote by $\L_n$ the free semigroup algebra generated by $L_1,\ldots,L_n$. This algebra is called the noncommutative Toeplitz algebra. It is the
	natural noncommutative analogue of $H^\infty$. We denote by $\A_n$ the norm-closed, unital algebra
	generated by $L_1,\ldots,L_n$. This algebra, called the \emph{noncommutative disc algebra}, was 
	introduced by Popescu \cite{Pop3}.
\end{example}

An isometry can be uniquely separated into the direct sum of a unilateral shift and a unitary. This is called the Wold decomposition of an isometry. A unitary can be further broken down into the direct sum of singular unitary and an absolutely continuous unitary, see e.g. \cite{Nagy}. Similarly a row-isometry has a Wold-type decomposition due to Popescu. A row-isometry can be written as the direct sum of an ampliation of the left-regular representation (the shift part) and a Cuntz-type row-isometry (the unitary part) \cite{Pop2}. Recently, Kennedy has shown that, like a single isometry, a row-isometry of Cuntz-type can be broken down further \cite{Ken2}. The decomposition is however into three parts: an absolutely continuous part, a singular part and a part of dilation-type. This is known as the \emph{Wold-von Neumann-Lebesgue decomposition} of a row-isometry. We define these terms here:
\begin{enumerate}
	\item a row-isometry is \emph{absolutely continuous} or \emph{analytic} if the free semigroup algebra it generates is isomorphic to $\L_n$,
	\item a row-isometry is \emph{singular} if the free semigroup algebra it generates is self-adjoint,
	\item a row-isometry is of \emph{dilation-type} if it has no absolutely continuous or singular
	direct summands.
\end{enumerate}
Some remarks on these terms are required. Firstly, that a row-isometry can be singular is not obvious. However, Read \cite{Read}, see also \cite{Dav06}, has shown the existence of a row-isometry which generates $\B(\H)$ as a free semigroup algebra.

Secondly, the term dilation-type can seem a bit obtuse. However it can be shown that a row-isometry of dilation type is necessarily the minimal isometric dilation of a defect-free row-contraction \cite{DLP, Ken2}. This means that if $S=[S_1,\ldots,S_n]$ is a row-isometry of dilation type on a Hilbert space $\H$ then there is a subspace $\V$ in $\H$ such that
  \begin{enumerate}
   \item $S_i^*\V\subseteq\V$ for each $i=1,\ldots,n$,
   \item $\H=\bigvee_{w\in\bF_n^+}S_w\V$.
  \end{enumerate}
When $\V$ is finite-dimensional then we say that $S$ is \emph{finitely correlated}. For more on isometric dilations of row-operators see \cite{DKS}. 

\subsection{\texorpdfstring{Free semigroup algebras in nonself-adjoint $2$-graph algebras}{Free semigroup algebras in nonself-adjoint 2-graph algebras}}
Let $\fS$ be a nonself-adjoint $2$-graph algebra generated by an isometric representation $(S,T)$ of a $2$-graph $\gr$, where $\theta\in S_{m\times n}$ with either $m>1$ or $n>1$. Note that for each $k,l\ge 0$ the set of operators $\{(ST)_w: d(w)=(k,l)\}$ are a family of isometries satisfying the Cuntz relations. We write 
$[ST]_{k,l}$ for this row-isometry and $\fS_{k,l}$ for the free semigroup algebra it generates. Note that the family of free semigroup algebras $\{\fS_{k,l}\}_{k,l\ge 0}$ span a dense subset of the nonself-adjoint $2$-graph algebra $\fS$.

\section{Analyticity and Rigidity}\label{sec: analytic}
\noindent Kribs and Power have shown that the nonself-adjoint $2$-graph algebra $\L_\theta$ is the natural higher-rank noncommutative analogy of $H^\infty$ and the free semigroup algebra $\L_n$ \cite{KriPow06}. In this section we consider $2$-graph algebras which are weak$^*$-weak$^*$ homeomorphic to $\L_\theta$. We call these algebras analytic. For a large class of $2$-graph algebras, those arising from what we call rigid representations, we will give a more algebraic definition of what it means for a $2$-graph algebra to be analytic.

\begin{defn}
	Let $(S,T)$ be an isometric representation of the $2$-graph $\gr$ and let $\fS$
	be the corresponding nonself-adjoint $2$-graph algebra. We say that $\fS$ is \textit{analytic} if it is completely isomorphic and weak$^*$ to weak$^*$ homeomorphic to $\L_\theta$. 
\end{defn}

In the case of free semigroup algebras, a free semigroup algebra $\fS$ is weak$^*$-weak$^*$ homeomorphic to $\L_n$ precisely when there is an injective \textsc{wot}-continuous homomorphism from $\fS$ to $\L_n$ \cite{DKP}. Hence, the existence of an injective \textsc{wot}-continuous homomorphism into $\L_n$ is used to define when a free semigroup algebra $\fS$ is analytic. 

A key factor in the simplicity of the characterisation of a free semigroup algebra being analytic is the rigidity of the norm-closed algebra generated by a row-isometry. That is, if $[S_1,\ldots,S_n]$ is a row-isometry then $\overline{\alg}^{\|\cdot\|}\{I,S_1,\ldots,S_n\}$ is completely isometrically isomorphic to the noncommutative disc algebra $\A_n$ \cite{Pop}.
There is more variation in the norm-closed algebras generated by representations of $2$-graphs. However, we will see in Theorem \ref{thm: rigidity} that in a wide class we do have the same rigidity as in the case of a single row-isometry. We give the following definition to describe these representations.

\begin{defn}
	Let $(S,T)$ be an isometric representation of $\gr$. We say that the representation $(S,T)$ is
	\emph{rigid} if $\overline{\alg}^{\|\cdot\|}\{I,S_1,\ldots,S_m,T_1,\ldots,T_n\}$ is completely
	isometrically isomorphic to $\A_\theta$. 
\end{defn}

We will first give an example of an isometric representation of a $2$-graph $\gr$ which is not rigid. We note that in this example $\bF_\theta^+$ is aperiodic.

\begin{example}\label{ex: not rigid}
	Let $[L_1,\ldots,L_n]$ be the left-regular representation of $\bF_n^+$ as in Example \ref{eg: fsg left-reg}. Let $[R_1,\ldots,R_n]$ be the right-regular representation, i.e. 
	\[
	 R_i\xi_w=\xi_{wi}
	\]
	for each $w\in\bF_n^+$ and $1\leq i\leq n$. It is clear that $L_iR_j=R_jL_i$ for $1\leq i,j\leq n$, and hence
	$(L,R)$ forms a representation of $\bF_{\id}^+$, where $\id$ is the identity permutation in
	$S_{n\times n}$. Let $\fA$ be the unital, norm-closed algebra generated by $(L,R)$. We will show
	that $\fA$ is not completely isometrically isomorphic to $\A_\id$.

	The representation $(L,R)$ extends to a completely contractive representation of $\A_\id$ if and only if it extends to a Cuntz-type representation of $\bF^+_\id$ \cite{DPYdiln}. Thus, it
	suffices to show that $(L,R)$ is not the compression of a Cuntz-type representation.

	Suppose $(S,T)$ is a Cuntz-type representation of $\bF_{\id}^+$ which extends $(L,R)$, i.e.
	$S_i|_{\H_n}=L_i$ and $T_j|_{\H_n}=R_j$. Denote by $\mt$ the identity element in $\bF_n^+$. 
	Note that
	\begin{equation*}
		R_2^*L_2L_1^*R_1\xi_{\mt}=\xi_{\mt}
	\end{equation*} 
	and
	\begin{equation*}
		R_2^*L_2L_1^*R_1\xi_{\mt}=T_2^*S_2S_1^*T_1\xi_{\mt}.
	\end{equation*}
	However, a calculation shows that any Cuntz-type representation of $\bF_{\id}^+$ necessarily satisfies
	$S_i^*T_j=T_jS_i^*$ for all $1\leq i,j\leq n$. Hence
	\begin{equation*}
		T_2^*S_2S_1^*T_1=S_1T_2^*T_1S_1^*=0,
	\end{equation*}
	since $T_2^*T_1=0$. This is a contradiction. Thus $(L,R)$ has no Cuntz-type extension and $\fA$ is not
	completely isometrically isomorphic to $\A_\id$.
\end{example}

As noted in Example \ref{ex: not rigid}, the norm-closed algebra of an isometric representation of $\gr$ being a completely contractive representation of $\A_\theta$ depends on the existence of a Cuntz-type extension. We now show the converse for the case when $\bF_\theta^+$ is aperiodic.

\begin{thm}\label{thm: rigidity}
	Let $(S,T)$ be a Cuntz-type representation of an aperiodic $2$-graph $\gr$ on $\H$. Let
	$\fA=\overline{\alg}^{\|\cdot\|}\{I,S_1,\ldots,S_m,T_1,\ldots,T_n\}$. Then, given any
	$\fA$-invariant projection $P$ on $\H$, the operator algebra $P\fA P$ is completely isometrically
	isomorphic to $\A_\theta$. That is, the representation $(PSP,PTP)$ is rigid.
\end{thm}

\begin{proof}
	We will first show that $\fA$ and $\A_\theta$ are completely isometrically isomorphic. The rest of the
	proof will follow from this, once we set up the correct commutative diagram.

	Since $(S,T)$ is a Cuntz-type representation it defines a $C^*$-representation of $\O_\theta$, the universal $C^*$-algebra for Cuntz-type representations of $\gr$, namely, the graph C*-algebra of $\gr$. In fact,
	since $\gr$ is aperiodic we have that $C^*(S,T)\cong\O_\theta$ \cite{DY1}. Further, note that
	$\O_\theta\cong C^*_{env}(\A_\theta)$ \cite{DPYdiln}. Hence, $\A_\theta$ canonically sits  completely isometrically
	inside $\O_\theta$, coinciding with $\fA$. It follows now that $\fA$ and $\A_\theta$ are completely isometrically isomorphic.

	Now let $P$ be any projection on $\H$ such that $P\H$ is an invariant subspace for $\fA$. Define
	representations $\pi$ and $\pi_P$ of $\A_\theta$ by $\pi(E_i)=S_i$ and $\pi(F_j)=T_j$; $\pi_P(E_i)=PS_i P$ 		and $\pi_P(F_j)=PT_j P$. By the preceding paragraph $\pi$ defines a completely isometrically isomorphic 	representation of $\A_\theta$. Let $\fI$ be the ideal
	of the $C^*$-algebra $C^*(P\fA P)$ generated by
	$P-\sum_{i=1}^m S_iPS_i^*$ and $P-\sum_{j=1}^n T_jPT_j^*$. We have that both $C^*(P\fA P)/\fI$ and 
	$C^*(S,T)$ are isomorphic to $\O_\theta$, and are hence isomorphic to each other. Denote by $p$ the
	natural isomorphism between $C^*(S,T)$ and $C^*(P\fA P)/\fI$, which sends generators to
	generators. We have the following commutative diagram:

	\[
		\xymatrix{
		\A_\theta \ar[r]|-{\pi} \ar[dd]_{\pi_P} &\fA\subseteq C^*(\pi(\gr)) &  \\
		 & & C^*(P\fA P)/\fI \ar[lu]_{p} \\
		P\fA P \ar@{^{(}->}[r]_{\id} &C^*(P\fA P) \ar[uu]_{p\circ q=:\tilde{\pi}} \ar[ur]_q &
		}
	\]
	Hence we have that
	\[
		\pi=\tilde{\pi}\circ\id\circ\pi_P.
	\]
	Now, as previously discussed, $\pi$ is a completely isometric isomorphism; $\id$ is a complete
	isometry; $\tilde{\pi}$ is completely contractive, since it is a $C^*$-homomorphism; and $\pi_P$ is
	completely contractive by \cite[Theorem 3.8]{DPYdiln}. Hence $\pi_P$ is a completely isometric
	isomorphism.
\end{proof}

Recall that we are primarily interested in nonself-adjoint $2$-graph algebras which arise from Cuntz-type representations of aperiodic $2$-graphs. By Theorem \ref{thm: rigidity}, all these nonself-adjoint $2$-graph algebras and their restrictions to invariant subspaces are rigid. Thus, we are not adding any conditions to our primary case of study when assume that our representations are rigid.

We will now state a number of results on analyticity. The assumption of rigidity allows one to follow methods similar to those in \cite{DKP}. We leave the details to the reader.

\begin{thm}\label{thm: analytic}
	Let $(S,T)$ be a rigid representation of $\gr$ with associated $2$-graph algebra $\fS$. Suppose
	$\varphi:\fS\rightarrow \L_\theta$ is a \textsc{wot}-continuous homomorphism such
	that $\varphi(S_i)=E_i$ and $\varphi(T_j)=F_j$. Then $\varphi$ is surjective and $\fS/\ker(\varphi)$
	is completely isometrically isomorphic to $\L_\theta$. Moreover, this map is a weak$^*$-weak$^*$
	homeomorphism.
\end{thm}

This theorem immediately allows us to give a simpler characterisation of when a nonself-adjoint $2$-graph algebra is analytic.

\begin{cor}\label{cor: analytic}
	Let $(S,T)$ be a rigid representation of $\gr$ with associated $2$-graph algebra $\fS$. Suppose there is
	an injective \textsc{wot}-continuous homomorphism $\varphi:\fS\rightarrow\L_\theta$ such that
	$\varphi(S_i)=E_i$ and $\varphi(T_j)=F_j$. Then $\fS$ is analytic.
\end{cor}

The following lemma follows the same line of proof as \cite[Corollary 1.3]{DKP}.

\begin{lem}\label{lem: join analytic}
	Let $\fS$ be a nonself-adjoint $2$-graph algebra acting on a Hilbert space $\H$. Denote by
	$\S$ the collection of all projections $S$ on $\H$ such that $\fS|_{S\H}$ is analytic.
	Let $Q$ be 	
	the projection
	\begin{equation*}
		Q=\bigvee_{S\in\S}S.
	\end{equation*}
	Then $\fS|_{Q\H}$ is analytic.
\end{lem}

\begin{defn}
 Let $(S,T)$ be an isometric representation of a $2$-graph $\gr$ on a Hilbert space $\H$ and let $\fS$ be the nonself-adjoint $2$-graph algebra generated by $(S,T)$. A unit
 vector $\zeta\in\H$ is called \emph{wandering} for $(S,T)$ and $\fS$ if $\<(ST)_u\zeta,(ST)_w\zeta\>=\delta_{u,w}$ for all $u,w\in\gr$.
\end{defn}

If a nonself-adjoint $2$-graph algebra $\fS$ has a wandering vector $\zeta$ then it follows from Corollary
\ref{cor: analytic} that $\fS|_{\fS[\zeta]}$ is analytic. In fact the homomorphism
$\varphi:\fS|_{\fS[\zeta]}\rightarrow \L_\theta$ is implemented by the unitary $U(ST)_w\zeta=\xi_w$. This fact,
together with Corollary \ref{cor: analytic} gives the following corollary to Lemma \ref{lem: join analytic}.

\begin{lem}\label{lem: wandering}
	Let $\fS$ be a nonself-adjoint $2$-graph algebra generated by a rigid representation.
	Suppose that $\zeta_j$ for $j\in\J$ are wandering vectors for $\fS$. Let $\M_j = S[\zeta_j]$
	and suppose that $\H = \bigvee_{j\in\J}\M_j$. Then $\fS$ is analytic and completely isometrically  		isomorphic to $\L_\theta$.
\end{lem}

Lemma \ref{lem: wandering} illustrates a connection between the existence of wandering vectors for a nonself-adjoint $2$-graph algebra $\fS$ and $\fS$ being analytic. In Section \ref{sec: atomic} we give some examples which have wandering vectors.

In the case of free semigroup algebras it has been shown that a free semigroup algebra is analytic precisely when the span of its wandering vectors is dense \cite{Ken}. It is unknown if the connection between wandering vectors and analyticity runs as deep for nonself-adjoint $2$-graph algebras.


\section{The Structure Theorems}
\noindent In this section we establish a lower-triangular $3\times 3$ form for rigid nonself-adjoint $2$-graph
algebra $\fS$. We first establish a $2\times 2$ form for any nonself-adjoint $2$-graph algebra. This form will be induced by a projection $P$ in $\fS$. Further, we will show that if $\fM$ is
the von Neumann algebra generated by $\fS$, then $\fS P=\fM P$. Thus, the left-hand column of the lower-triangular structure of $\fS$ will be a slice of the enveloping von Neumann algebra. This structure decomposition of $\fS$ depends deeply on the structure of the individual free semigroup algebras $\fS_{k,l}$ inside $\fS$. In the case of finitely correlated representations the desired projection has previously been described in \cite{Ful}, however its relationship to $\fM$ is new to this paper. We will be able to describe the projection similarly when certain free semigroup algebras inside $\fS$ are generated by row-isometries of dilation type. In general, however, the description of the projection is not as simple. In section \ref{sec: 2nd proj} we further decompose the $(2,2)$-entry of the $2\times 2$ form of $\fS$ when $\fS$ is rigid. This will give us a $3\times 3$ form. Here, the $(3,3)$-entry will be an analytic $2$-graph algebra.

\subsection{\texorpdfstring{A lower-triangular $2\times 2$ decomposition}{A lower-triangular 2 by 2 decomposition}}

In this subsection, we prove that every nonself-adjoint 2-graph algebra $\fS$ has a $2\times 2$ low-triangular form with the aid of its first structure projection $P$. 
Moreover, the left-hand column of $\fS$ is a slice of its enveloping von Neumann algebra. 

As the structure of free semigroup algebras plays an important role in the following analysis, we first recap the necessary results for free semigroup algebras here. The structure projection of a free semigroup algebra was introduced by Davidson, Katsoulis and Pitts \cite{DKP}. Earlier, in the case of finitely correlated representations the phenomenon was observed by Davidson, Kribs and Shpigel \cite{DKS}. The following theorem summarises some of their results \cite[Theorem 2.6, Corollary 2.7]{DKP}.

\begin{thm}[Properties of the Structure Projection]\label{struct props}
	Let $\fT$ be a free semigroup algebra. Then $\fT$ contains a projection with the following properties:
	\begin{enumerate}
		\item[(i)] $P\fT P$ is self-adjoint,
		\item[(ib)] $\fT P=\fM P$, where $\fM$ is the von Neumann algebra generated by $\fT$,
		\item[(ii)] $P^\perp\H$ is invariant for $\fT$,
		\item[(iii)] $P^\perp$ is the largest projection $Q$ so that $\fT_{Q\H}$ is analytic.
	\end{enumerate}
	Furthermore, the projection $P$ is unique.
\end{thm}

\begin{defn}
 The projection in Theorem \ref{struct props} is called the \emph{structure projection} of the free semigroup algebra.
\end{defn}

 When $\fS$ is a nonself-adjoint $2$-graph algebra we denote by $P_{k,l}$ the structure projection of the free semigroup algebra $\fS_{k,l}$.

The case of finitely correlated nonself-adjoint $2$-graph algebras was studied in \cite{Ful}. Here we have for $k,l>0$ that $P_{k,l}=P_{1,1}=:P$ \cite{Ful}. In this case it is tempting to say $P$ is a structure projection for
$\fS$. Indeed it is shown in \cite{Ful} that $P$ satisfies the conditions (i) and (ii) of Theorem \ref{struct props}. This inspires the following definition.

\begin{defn}
	Let $\fS$ be a nonself-adjoint $2$-graph algebra.
	Let $P_{k,l}$ be the structure projection for $\fS_{k,l}$. We define the \emph{first structure projection} $P$ of $\fS$ to be
	\begin{equation*}
		P=\bigwedge_{k,l>0}P_{k,l}.
	\end{equation*}
\end{defn}

	We will show that the first structure projection satisfies analogues of properties (i), (ib) and (ii) 		of Theorem \ref{struct props}. It may not however satisfy the desired
	analogue of property  (iii). This leads us to consider a second structure projection
	for nonself-adjoint aperiodic $2$-graph algebras. We discuss this further in subsection \ref{sec: 2nd proj}. In section \ref{sec: atomic} we give examples of nonself-adjoint $2$-graph algebras with an invariant subspace on which it acts analytically. We do this by showing the existence of wandering vectors.

The following lemma uses methods similar to those in \cite[Theorem 3.19]{Ful}. Here
we use the existence of wandering vectors for analytic free semigroup algebras \cite{Ken} to overcome the fact that our representation is not finitely correlated.

\begin{lem}\label{struct proj ii}
Suppose that $(S,T)$ is a Cuntz-type representation of $\Fth$. Let $P_{k,l}$ be the structure
projection for $\fS_{k,l}$, where $k,l>0$ and let $\V = P_{k,l}\H$. Then $\V$ is $\fS^*$-invariant.

Thus if $P$ is the first structure projection for $\fS$, then $P\H$ is invariant for $\fS^*$.
\end{lem}

\begin{proof}
Define a subspace $\M$ of $\H$ by
\begin{equation*}
	\M = \sum_{d(w)=(k,l-1)}(ST)_w^*\V.
\end{equation*}
It follows that $T_j^*\M\subseteq\V$ for $j = 1,\ldots,n$. Thus if we can show that $\V\subseteq\M$ then it will follow that $T_j^*\V\subseteq\V$.

Note that we also have $(ST)_u^*\M\subseteq\M$, when $d(u)=(k,l)$.
It follows that $T_j\V^\perp\subseteq\M^\perp$ and $(ST)_u\M^\perp \subseteq\M^\perp$, when $d(u)=(k,l)$. By \cite{Ken} there are vectors in $\V^\perp$ which are
wandering for $\fS_{k,l}$. We will show that if $w\in\V^\perp$ is wandering for $\fS_{k,l}$ then $T_jw$
is wandering for $\fS_{k,l}$:
Take $w\in\V^\perp$ which is wandering for $\fS_{k,l}$. Suppose that $u\in\gr$, $d(u)=(pk,pl)$ for some $p\geq1$. By the commutation relations $(ST)_u=T_{j'}(ST)_{u'}$ for some $j'$ and some word $u'$ with $d(u')=(pk,pl-1)$. If $j'\neq j$ then it follows immediately that
\begin{equation*}
	\<(ST)_uT_j w,T_j w\>=0.
\end{equation*}
If $j'=j$ then 
\begin{equation*}
	\<(ST)_u T_j w,T_j w\>=\<(ST)_{u'}T_j w, w\> = 0
\end{equation*}
since $w$ is wandering for $\fS_{k,l}$. Hence $T_jw$ is $\fS_{k,l}$-wandering.

We have shown that the set $\W$ of $\fS_{k,l}$-wandering vectors in $\M^\perp$ is non-empty. We define the $\fS_{k,l}$-invariant subspace $\L\subseteq\M^\perp$
\begin{equation*}
	\L=\bigvee_{w\in\W}\fS_{k,l}[w].
\end{equation*}
Since, by \cite{DKP,Ken}, if $\W'$ is the set of $\fS_{k,l}$-wandering vectors in $\V^\perp$,
\begin{equation*}
	\V^\perp=\bigvee_{w\in\W'}\fS_{k,l}[w],
\end{equation*}
it follows by above that $T_jh\in\L$ for all $h\in\V^\perp$.

Take $x\in\M^\perp\ominus\L$, then we have shown that $(ST)_ux\in\L$ when $d(u)=(pk,pl)$. In particular $(ST)_ux\perp x$, i.e. $x$ is $\fS_{k,l}$-wandering. This contradicts the
choice of $x$. Hence $\L = \M^\perp$, and $\fS_{k,l}|_{\M^\perp}$ is analytic. 
By Theorem \ref{struct props}, it follows that $\V\subseteq\M$. Since
$T_j^*\M\subseteq\V$ it follows that $T_j^*\V\subseteq\V$.

A similar argument, with $\N = \sum_{d(w)=(k-1,l)}(ST)_w^*\V$ in place of $\M$, will show that $\V$ is $\fS^*$-invariant.
\end{proof}

\begin{lem}\label{struct proj i}
	Let $P$ be the first structure projection of $\fS$. Then $P\fS P$ is a self-adjoint algebra.
\end{lem}

\begin{proof}
	Take $k,l>0$. If $A\in\fS_{k,l}$ then $P_{k,l}A^*P_{k,l}\in P_{k,l}\fS P_{k,l}$, since $P_{k,l}\fS_{k,l}P_{k,l}$ is self-adjoint. Thus we have
	\begin{align*}
		PA^*P&=PP_{k,l}A^*P_{k,l}P\\
			&\in PP_{k,l}\fS P_{k,l}P=P\fS P.
	\end{align*}
	
	Now suppose we have $A\in\fS_{k,0}$. Then, since our representation is Cuntz-type and $P\H$ is $\fS^*$-invariant, we have
	\begin{align*}
		PAP&=\sum_{i,j}PAS_iT_jT_j^*S_i^*P\\
			&=\sum_{i,j}(PAS_iT_jP)(PT_j^*S_i^*P).
	\end{align*}
	Now for each $i,j$, $AS_iT_j$ lies in a subalgebra of $\fS$ which has the  algebra of polynomials
	\begin{equation*}
		\P_{k+1,1}=\left\{\sum_{a>k,\ b>0}A_{a,b}:\ A_{a,b}\in\fS_{a,b}, A_{a,b}=0\text{ for all but finitely many } a,b \right\}.
	\end{equation*}
	as a $\wot$-dense subalgebra. Since, by above, each element $X\in\P_{k+1,1}$ satisfies $PX^*P\in P\fS P$ it follows that $P(AS_i T_j)^*P\in P\fS P$. Thus
	\begin{align*}
		PA^*P&=\left(\sum_{i,j}^n(PAS_iT_jP)(PT_j^*S_i^*P)\right)^*\\
			&=\sum_{i,j}(PS_iT_jP)(P(AS_iT_j)^*P)\in P\fS P.
	\end{align*}
	Similarly if $A\in \fS_{0,l}$ we can show $PA^*P\in P\fS P$.
	
	Arguing in a similar manner to above, the polynomials
	\begin{equation*}
		\P=\left\{\sum_{k,l\geq0}A_{k,l}:\ A_{k,l}\in\fS_{k,l}, A_{k,l}=0\text{ for all but finitely many } k,l\geq0\right\}
	\end{equation*}
	forms  a dense algebra in $\fS$. Thus $P\fS P$ is self-adjoint.
\end{proof}

We summarise these results in the following theorem, adding property (ib) of Theorem \ref{struct props} that the first structure projection defines 
a slice of the enveloping von Neumann algebra for $\fS$.

\begin{thm}\label{thm: struct thm}
	Let $\fS$ be a nonself-adjoint $2$-graph algebra associated to a Cuntz-type representation $(S,T)$ of 
	$\Fth$, and let $\fM$ be the enveloping
	von Neumann algebra. Then the first structure projection $P$ satisfies the following properties:
	\begin{enumerate}
		\item $P\fS P$ is self-adjoint,
		\item $P^\perp\H$ is an $\fS$-invariant subspace, and
		\item $\fS=\fM P+P^\perp\fS P^\perp$.
	\end{enumerate}
\end{thm}

\begin{proof}
	It is only left to show the final part (iii). In order to do this we will show that the 
	\textsc{wot}-closed left ideal $\fJ=\fS P$ is also a left ideal in $\fM$. That is,
	we will show that $\fS P=\fM P$. 

	First note that
	\begin{equation*}
		\spn\{T_\alpha S_\beta T_\gamma^* S_\delta^*: \alpha,\gamma\in\bF_n^+, \beta,\delta\in\bF_m^+\}
	\end{equation*}
	is a dense subset in $\fM$. This follows from the fact that the representation $(S,T)$ is of Cuntz-type.
	For any $\alpha,\gamma\in\bF_n^+$ and $\beta,\delta\in\bF_m^+$ we have
	\begin{align*}
		T_\alpha S_\beta T_\gamma^* S_\delta^*P&=T_\alpha S_\beta PT_\gamma^* S_\delta^*P
	\end{align*}
	since $P\H$ is $\fS^*$-invariant. We also have $PT_\gamma^* S_\delta^*P\in P\fS P$ since
	$P\fS P$ is self-adjoint by Lemma \ref{struct proj i}.
	Hence
	\begin{equation*}
		T_\alpha S_\beta T_\gamma^* S_\delta^*P\in\fJ,
	\end{equation*}
	and $\fM P=\fS P$.
\end{proof}

\subsection{More on the first structure projection}\label{sec: 1st proj}
As discussed previously, in the case of a finitely correlated representation the structure projections of each
$\fS_{k,l}$ when $k,l>0$ all coincide. We prove a similar result here. We will show that if $(S,T)$ is a Cuntz-type representation of $\gr$ and $[ST]_{k,l}$ is of dilation type for all $k,l>0$ then we
have that $P_{k,l}=P_{1,1}$ for all $k,l>0$. Hence this shared structure projection for the free semigroup algebras is the first structure projection for the nonself-adjoint $2$-graph algebra.

\begin{lem}\label{lem: diln type}
	If the row-isometry $[(ST)_w: d(w)=(k,l)]$ is of dilation type for some $k,l>0$ then $P_{k,l}\geq P_{p,q}$ for all $p,q>0$.
\end{lem}

\begin{proof}
 	Let $\V=P_{k,l}\H$. By \cite[Proposition 6.2]{Ken2}, $[(ST)_w: d(w)=(k,l)]$ is the minimal isometric dilation of the compression of
	$[(ST)_w: d(w)=(k,l)]$ to $\V$. We know by Lemma \ref{struct proj ii} that $\V$ is invariant under $\fS^*$, and hence the compressions
	of each $S_i$ and $T_j$ to $\V$ form a representation of $\gr$. This representation will be defect-free, i.e. coisometric (see e.g.
	 \cite[Lemma 3.10]{Ful}.) We will denote the compression of $S$ to $\V$ by $A$ and the compression of $T$ to $\V$ by $B$. We wish to show
	 that the joint minimal isometric dilation of $(A,B)$ is $(S,T)$.
	 
	 Suppose $(\hat{S},\hat{T})$ is the unique minimal isometric dilation of $(A,B)$ and $(S,T)$ is not minimal.
	 By construction $(S,T)$ is an isometric dilation of $(A,B)$ and so,
	 by the minimality of $(\hat{S},\hat{T})$,
	 \[
	 	S_i=\hat{S}_i\oplus\hat{S}_i'
	 \]
	 and
	 \[
	 	T_j=\hat{T}_j\oplus\hat{T}_j'
	 \]
	 for each $1\leq i\leq m$ and $1\leq j\leq n$ where $(\hat{S}',\hat{T}')$ is some Cuntz-type representation of $\gr$. 
	 
	 Now, by the uniqueness of minimal isometric dilations of row-contractions and \cite[Theorem 3.12]{Ful}, we have that when $k,l>0$
	 \[
	      (ST)_w=(\hat{S}\hat{T})_w
	 \]
	 when $d(w)=(k,l)$. Hence the isometry $(\hat{S}'\hat{T}')_w=0$ when $d(w)=(k,l)$. This contradiction tells us that $(S,T)$ is the minimal isometric dilation of $(A,B)$.
	 
	 Now take any $p,q>0$. Another application of  \cite[Theorem 3.12]{Ful} tells us that $[(ST)_w: d(w)=(p,q)]$ is the minimal isometric dilation of 
	 $[(AB)_w: d(w)=(p,q)]$. It follows now, by \cite[Lemma 3.1]{DKS} and \cite[Corollary 2.7]{DKP} that $P_{p,q}\leq P_{k,l}$.
\end{proof}

We can now immediately describe the first structure projection in the case that $[(ST)_w: d(w)=(k,l)]$ is of dilation type for all $k,l>0$. This generalises what was
already known for finitely correlated representations (\cite[Proposition 4.12]{Ful}).

\begin{thm}\label{thm: diln type}
	Let $(S,T)$ be a Cuntz-type representation of $\gr$ and let $\fS$ be the nonself-adjoint $2$-graph
	generated by $(S,T)$.
	If the row-isometries $[(ST)_w: d(w)=(k,l)]$ are of dilation type for all $k,l>0$ then
	$P_{k,l}=P_{p,q}$ for all $k,l,p,q>0$.
	In particular the shared structure projection
	for the free semigroup algebras $\fS_{k,l}$, $k,l>0$, is the first structure projection for the
	nonself-adjoint $2$-graph algebra $\fS$.
\end{thm}

Theorem \ref{thm: diln type} hints at a dependence the structure of the row-isometries $[(ST)_w: d(w)=(k,l)]$ when $k,l>0$ have on each other. We explore this further in the following proposition, where we show that if one of these row-isometries is of dilation type then the others all have no singular part in their Lebesgue-von Neumann-Wold decomposition. In order to do this we first prove a technical lemma about row-isometries of dilation type.

\begin{lem}\label{lem: diln type 2}
	Let $T=[T_1,\ldots,T_n]$ be a row-isometry of dilation type on the Hilbert space $\H$.
	Let $\fT$ be the free semigroup algebra generated by $T$,let $Q$ be the structure projection for
	$\fT$ and let $\V=Q\H$. Then if $v\in\V$ is a non-zero vector, there exists some $w\in\bF_n^+$ so that
	$T_wv\notin\V$. 
\end{lem}

\begin{proof}
	Suppose there is a non-zero vector $v$ in $\V$ such that $T_wv\in\V$ for all $w\in\bF_n^+$. Then the 		subspace $\fT[v]$ is a $\fT$-invariant subspace
	of $\V$. Note that $\fT[v]$ is not equal to all of $\V$ since if it were then $\V$ would be a reducing
	subspace for $\fT$, contradicting the fact the $T$ is of dilation type. Let $\U=\V\ominus\fT[v]$. Then 		$\U$ is a $\fT^*$-invariant subspace. It follows that $\fT[\U]$
	is a $\fT$-reducing subspace. Since $T$ is of dilation type, we have that $\fT[\V]=\H$. It follows
	that $\V^\perp\subseteq\fT[\U]$.

	Now, if $\V\subseteq\fT[\U]$ then the structure projection for
	$\fT$ is at most the projection onto $\U$. Since $v\notin\U$ we must have that $\V\not\subseteq\fT[U]$.
	Hence the subspace $\U':=\V\ominus(\fT[\U]\cap\V)$ is non-empty. Since $\fT[\U]$ is $\fT$-reducing
	and $\V$ is $\fT^*$-invariant, the subspace $\U'$ is $\fT^*$-invariant and orthogonal to $\U$.
	Since $\V^\perp\subseteq\fT[\U]$ it follows that $\V=\U\oplus\U'$ by the definition of $\U'$ and since 		$Q$ is the structure projection for $\fT$. It follows now that $\U'\subseteq\V$ is a $\fT$-reducing
	subspace. This contradicts $T$ being of dilation type, since in this case we must have that $T$ is
	singular on $\U'$. 
\end{proof}

\begin{prop}\label{prop: diln type 2}
	Suppose $[(ST)_w: d(w)=(k,l)]$ is of dilation type for some $k,l>0$, then when $p,q>0$ the isometric
	tuple $[(ST)_w: d(w)=(p,q)]$ has no singular part in its Wold-von Neumann-Lebesgue decomposition.
\end{prop}

\begin{proof}
	Take any $p,q>0$. By Lemma \ref{lem: diln type} the structure projection corresponding to
	$[(ST)_w: d(w)=(p,q)]$ is dominated by the structure projection corresponding to $[(ST)_w: d(w)=(k,l)]$.
	It follows that, if $[(ST)_w: d(w)=(p,q)]$ has a singular part then this is realised on a 
	subspace $\U$ of $\V$.

	Take any non-zero vector $u\in\U$. By Lemma \ref{lem: diln type 2} there is a $w\in\gr$ such that
	$d(w)=(\alpha k,\alpha l)$ for some positive integer $\alpha$ and $(ST)_wu$ is not in $\V$. Now choose 
	a $v\in\gr$ such that $v=v'w$ for some $v'\in\gr$ and $d(v)=(\beta p,\beta q)$ for some positive
	integer $\beta$. By Lemma \ref{struct proj ii}, $\V^\perp$ is invariant for $S$ and $T$. It follows 		that $(ST)_vu$ is not in $\V$. In particular $(ST)_vu$ does not lie in $\U$, contradicting the
	assumption that $\U$ was reducing. Hence $[(ST)_w: d(w)=(p,q)]$ has no singular part.
\end{proof}

\subsection{\texorpdfstring{A lower-triangular $3\times 3$ decomposition}{A lower-triangular 3 by 3 decomposition}}\label{sec: 2nd proj}
As mentioned previously, one downside of Theorem \ref{thm: struct thm} when compared to the structure theorem for free semigroup algebras is that we do not know how the algebra behaves on the complement of the subspace given by the first structure projection.

Let $\fS$ be a nonself-adjoint $2$-graph algebra on $\H$ and let $P$ be its first structure projection. Unlike the free semigroup case, we know that $\fS|_{P^\perp\H}$ is not necessarily analytic. What if there was an invariant subspace $\M$ such that $\fS|_\M$ was analytic? Does Theorem \ref{thm: struct thm} give any insight into $\M$? In fact, it does.

\begin{lem}\label{lem: 2nd proj}
 Let $\fS$ be a nonself-adjoint $2$-graph algebra acting on a Hilbert space $\H$. Let $P$ be the first structure projection for $\fS$ as given in Theorem \ref{thm: struct thm}.

 Suppose that $\M\subseteq\H$ is an $\fS$-invariant subspace such that $\fS|_\M$ is analytic. Then $\M\subseteq P^\perp\H$.
\end{lem}

\begin{proof}
 As $\fS|_\M$ is analytic, for any $k,l\geq0$ $\fS_{k,l}|_\M$ is an analytic free semigroup algebra. Let $P_{k,l}$ be the structure projection for $\fS_{k,l}$. By \cite[Corollary 2.7]{DKP} $\M$ is orthogonal to $P_{k,l}\H$. Hence, since
 \[ P=\bigwedge_{k,l>0}P_{k,l}, \]
 $\M$ is orthogonal to $P\H$.
\end{proof}

With Lemma \ref{lem: join analytic} and Lemma \ref{lem: 2nd proj} in hand we are now ready to define the second structure projection.

\begin{defn}
	Let $\fS$ be a nonself-adjoint $2$-graph algebra acting on a Hilbert space $\H$. Let
	$\M\subseteq\H$ be the largest space such that $\fS|_{\M}$ is analytic. We call the 
	projection onto $\M$ the \emph{second structure projection} for $\fS$.
\end{defn}

We can now determine a $3\times 3$ decomposition of a nonself-adjoint $2$-graph algebra arising from a rigid representation.

\begin{thm}\label{thm: struct thm 2}
	Let $(S,T)$ be a rigid representation of a $2$-graph $\gr$. Let $\fS$ nonself-adjoint
	$2$-graph algebra generated by $(S,T)$ and let $\fM$ be the von-Neumann algebra
	generated by $\fS$.

	Then the nonself-adjoint $2$-graph algebra $\fS$ has a lower-triangular
	$3\times 3$ form. The left hand column in this decomposition is a left ideal $\fJ$ of $\fM$.
	The $(3,3)$-entry of the decomposition is an analytic nonself-adjont $2$-graph algebra.
\end{thm}

\begin{proof}
	Theorem \ref{thm: struct thm} determines a lower triangular $2\times 2$ structure for $\fS$ 
	with a left ideal $\fJ$ of $\fM$ as the left-hand column.

	Since $(S,T)$ is a rigid representation, Lemma \ref{lem: join analytic} tells us that there 		is a largest subspace on which $\fS$ acts analytically.
	Lemma \ref{lem: 2nd proj} tells us that the second structure projection is necessarily 		orthogonal to the first structure projection.
	Hence we can further decompose the $(2,2)$-entry in the lower-triangular $2\times 2$ form 		obtained in Theorem \ref{thm: struct thm} to obtain a lower $3\times 3$ form.
\end{proof}

In Section \ref{sec: atomic} we will show that the second structure projection is necessarily non-zero for some 2-graph algebras generated from  atomic representations. We do this by showing the existence of wandering vectors in those cases.

\section[von Neumann Algebras]{\texorpdfstring{$2$-Graph Algebras as von Neumann Algebras}{2-Graph Algebras as von Neumann Algebras}}
\noindent Though generated in a nonself-adjoint way, nonself-adjoint $2$-graph algebras can be von Neumann algebras. This is perhaps surprising, but it is less so when we recall that free semigroup algebras can be self-adjoint. Indeed the following proposition shows that all free semigroup algebras can be described as
nonself-adjoint $2$-graph algebras. In this section we will also discuss the role the free semigroup algebras contained in a nonself-adjoint $2$-graph algebra play in determining if a nonself-adjoint $2$-graph algebra is a von Neumann algebra.

\begin{prop}\label{prop: fsg=2 graph alg}
	Let $\fS$ be a free semigroup algebra. Then $\fS$ can be described as a nonself-adjoint periodic 
	$2$-graph algebra.
\end{prop}

\begin{proof}
	Let $S=[S_1,\ldots,S_n]$ be the row-isometry which generates $\fS$. Let $\alpha\in S_n$ be a
	permutation on $n$ elements. Let $T$ be the row-operator defined by
	\begin{equation*}
		T=[S_{\alpha(1)}, S_{\alpha(2)},\ldots, S_{\alpha(n)}].
	\end{equation*}
	Then $S$ and $T$ satisfy the commutation relations:
	\begin{equation}\label{calc1}
		S_iT_j=S_iS_{\alpha(j)}=T_{\alpha^{-1}(i)}S_{\alpha(j)}.
	\end{equation}
	Thus $(S,T)$ is a representation of $\gr$ where $\theta(i,j)=(\alpha(j),\alpha^{-1}(i))$. The
	nonself-adjoint $2$-graph algebra generated by $(S,T)$ is clearly equal to $\fS$. That $\gr$ is 
	a periodic $2$-graph follows immediately from \eqref{calc1} above and Theorem \ref{thm: periodic cond}.
\end{proof}

Taking $S$ in Proposition \ref{prop: fsg=2 graph alg} to be a singular row-isometry gives an example of a
nonself-adjoint $2$-graph algebra which is a von Neumann algebra. This example, however, is generated by a representation of a \textit{periodic} $2$-graph. The following example shows that there are examples which are generated by \textit{aperiodic} $2$-graph algebras.

\begin{example}
Let $S=[S_1,...,S_n]$ acting on $\H$ be a singular row isometry such that  the free semigroup algebra $\fS$ generated by $S$ is $\B(\H)$.
Then $U=[S_1\otimes I, ..., S_n\otimes I]$ and $V=[I\otimes S_1,..., I\otimes S_n]$ give a Cuntz-type representation of the $2$-graph $\bF_\id^+$. It is easy to see that the nonself-adjoint $2$-graph algebra
$\fS$ generated by $U$ and $V$ is $\B(\H)\otimes\B(\H)$, which is  a free semigroup algebra.
\end{example}

Theorem $1.5$ of \cite{DKP} establishes necessary and sufficient conditions for a free semigroup algebra to be a von Neumann algebra. If $\fT$ is a free semigroup generated by 
a row-isometry $[T_1,\ldots,T_n]$ then
$\fT$ is a von Neumann algebra if and only if the ideal in $\fT$ generated by $T_1,\ldots,T_n$ contains the identity.

The case of nonself-adjoint $2$-graph algebra $\fS$ is more complicated. Simply looking at the ideal generated by the generators of $\fS$ will not determine if $\fS$ is a von Neumann algebra.
We will however establish conditions for a nonself-adjoint $2$-graph algebra $\fS$ to be a von Neumann algebra based on other ideals in $\fS$.

But a simple test first. 

\begin{prop}
\label{P:ss}
Let $\fS$ be a nonself-adjoint 2-graph algebra generated by a representation $(S,T)$. If both $S$ and $T$ are singular, then $\fS$ is self-adjoint. 
\end{prop}

\begin{proof}
Since $S=[S_1,...,S_m]$ is a singular row-isometry, $\fS_{1,0}$ is a self-adjoint free semigroup algebra. So $S_i^*\in\fS$ for $1\le i\le m$. Similarly,
if $T=[T_1,...,T_n]$, then $T_j^*\in\fS$ for $1\le j\le n$.  
\end{proof}

The following lemma comes from the proof of \cite[Lemma 2.1]{DKP}.

\begin{lem}\label{DKP lemma} If $\fJ$ is a \textsc{wot}-closed right-ideal in a \textsc{wot}-closed algebra $\fA$ generated by isometries $V_1,\ldots,V_k$ with pairwise orthogonal ranges, then every element $A\in\fJ$ can be written
uniquely as
	\begin{equation*}
		A=\sum_{i=1}^kV_iA_i
	\end{equation*}
where for each $i$, $A_i\in\fA$.
\end{lem}

\begin{thm}\label{thm: selfadj}
	Given $k,l>0$, let $\fJ_{k,l}$ be the \wot-closed right-ideal generated by the row-isometry
	$[(ST)]_{k,l}$ in a nonself-adjoint $2$-graph algebra $\fS$. Then $\fS$ is a von Neumann algebra if 		and only if $\fJ_{k,l}$ contains the identity.
\end{thm}

\begin{proof}
	Suppose that $I\in\fJ_{k,l}$. Then by Lemma \ref{DKP lemma} there are $A_w\in\fS$ such that
	\begin{equation*}
		I=\sum_{\substack{w\in\gr\\ d(w)=(k,l)}}(ST)_wA_w.
	\end{equation*}
	Hence
	\[
		S_{i_0}^*=\sum_{w=i_0w'} (ST)_{w'}A_w
	\]
	is in $\fS$. Similarly $T_j^*\in\fS$ for each
	$j$. Thus $\fS$ is self-adjoint.

	Conversely suppose that $\fS$ is self-adjoint. Then
	\begin{equation*}
		I=\sum_{d(w)=(k,l)}(ST)_w(ST)_w^*\in\fJ.\qedhere
	\end{equation*}
\end{proof}

Proposition \ref{thm: selfadj} now gives a sufficient condition for a nonself-adjoint $2$-graph algebra $\fS$ to be a von Neumann algebra based on the free semigroup
algebras $\fS_{k,l}$ contained in $\fS$.

\begin{cor}\label{S_kl selfadj}
	If there are $k,l>0$ such that $\fS_{k,l}$ is self-adjoint then $\fS$ is self-adjoint.
\end{cor}

\begin{proof}
	Since $\fS_{k,l}$ is self-adjoint it follows from \cite[Theorem 1.5]{DKP} that
	$I\in\fJ_{k,l}$. Hence $\fS$ is self-adjoint by Theorem \ref{thm: selfadj}.
\end{proof}

The condition that both $k>0$ and $l>0$ in the previous corollary is necessary. The following example shows this, as well as showing how pathological nonself-adjoint $2$-graph algebras can be. It should be noted that the following example arises from a representation of a periodic $2$-graph.

\begin{example}
Let $S_1$ and $S_2$ be isometries with pairwise orthogonal range such
that $\ol\alg^\wot\{I, S_1, S_2\}$
is a von Neumann algebra. Define operators
	\begin{align*}
		U_1=S_1\otimes I && U_2=S_2\otimes I\\
		V_1= S_1\otimes U && V_2=S_2\otimes U
	\end{align*}
where $U$ is a bilateral shift. Then $[U_1,U_2]$ and $[V_1, V_2]$ form a representation of the $2$-graph
$\gr$, where $\theta$ is the flip permutation. Clearly the free semigroup algebra generated by $U_1$ and $U_2$ is self-adjoint. On the other hand, the free semigroup generated by $V_1$ and $V_2$ is analytic and hence it has no singular part. The
nonself-adjoint $2$-graph algebra generated by $U_1,U_2,V_1,V_2$ is isomorphic to $\B(\H)\otimes H^\infty$, which is not self-adjoint.
\end{example}

\section{Atomic Representations}\label{sec: atomic}
\noindent We will now focus on atomic representations. We will see that even in this simple class of examples it is not necessarily possible to find wandering vectors. One may expect that the finitely correlated case would be the simplest, however in this case we are not guaranteed wandering vectors. In some cases, however, we will be able to specifically find wandering vectors.

To simplify our discussion,  throughout this section we assume that $m, n\ge 2$.  
We begin by recalling the definition of an atomic representation.

\begin{defn}
 A representation $(A,B)$ of a $2$-graph $\gr$ on $\H$ is \emph{atomic} if there is an orthonormal basis $\{\xi_k: k\geq 0\}$ for $\H$ such that given $w\in\gr$ and basis vector $\xi_i$ there is a scalar $\alpha\in\bT\cup\{0\}$ and a basis vector $\xi_j$ such that $(AB)_w\xi_i=\lambda_{w,i}\xi_j$. We call the basis $\{\xi_k: k\geq 0\}$ the \emph{standard basis} for the representation.
\end{defn}

The scalars in the above definition will not play an important role in our analysis. We will adopt the convention of writing $\ol{\xi}$ for $\{\lambda\xi:\ \lambda\in\bT\}$. Thus we will write $(AB)_w\ol{\xi_i}=\ol{\xi_j}$ to mean $(AB)_w\xi_i=\lambda_{w,i}\xi_j$.

An atomic representation determines a $2$-coloured directed graph with vertices $\{\ol{\xi_k}: k\geq 0\}$. We draw a blue edge from $\ol{\xi_i}$ to $\ol{\xi_j}$ if there is an $A_k$ so that $A_k\ol{\xi_i}=\ol{\xi_j}$. Similarly we draw a red edge from $\ol{\xi_i}$ to $\ol{\xi_j}$ if there is an $B_k$ so that $B_k\ol{\xi_i}=\ol{\xi_j}$. Thus a path from one vertex to another is determined by a word $w\in\gr$. We call this graph the \emph{graph of an atomic representation.}

Note that in the graph of a Cuntz-type atomic representation every vertex has exactly one red and one blue edge leading into it. From this we conclude that given $(k,l)$ every vertex has a \emph{unique} path $w$ of degree $(k,l)$ leading into it. Since a Cuntz-type representation is isometric we also have that each vertex necessarily has $m$ blue and $n$ red edges leading from it.

Atomic representations of free semigroup algebras were classified by Davidson and Pitts in \cite{DP}. The irreducible Cuntz-type representations were shown to be either of infinite tail type or of (finitely-correlated) ring type. In our setting, if $(S,T)$ is an atomic representation of $\gr$ then $S$ is an atomic representation of $\bF_m^+$ and $T$ is an atomic representation of $\bF_n^+$. In \cite{DPYatom} Davidson, Power and the second author classified the atomic representations of $2$-graphs. The irreducible Cuntz-type representations, broadly speaking, break into $3$ categories, based on the classifications of $S$ and $T$ separately. It is shown in \cite{DPYatom} that a Cuntz-type atomic representation necessarily has a cyclic coinvariant subspace $\V$ spanned by a subset of the standard basis for the representation associated to it. We summarise the different types of atomic representations and their corresponding cyclic coinvariant subspaces $\V$ here and refer the reader to \cite{DPYatom} (and \cite{DP}) for full details on the different types of atomic representations.

\hypersetup{bookmarksdepth=-3}
\subsection*{\texorpdfstring{Type 1}{}} In this case both $S$ and $T$ are ring type representations of free semigroups. In this instance $\V$ is finite dimensional. If we take a basis vector $\xi$ in $\V$ the there exist words $e_{u_0}\in\bF_m^+$ and $f_{v_0}\in\bF_n^+$ of minimal lengths so that
\[ S_{u_0}\ol\xi=\ol\xi=T_{v_0}\ol\xi. \]
Indeed $\V$ is determined by $e_{u_0}$, $f_{v_0}$ and $\xi$.

\subsection*{\texorpdfstring{Type 2}{}} In this case either $S$ is of ring type and $T$ is of infinite tail type (type 2a) or $T$ is of ring type and $S$ is of infinite tail type (type 2b).

A type 2a representation $(S,T)$ is determined by a word blue $u_0$, an infinite red tail $v_0=j_{0,−1}j_{0,−2}\dots$ and a sequence $0 > t_1 > t_2 > \ldots$ such that
	\[e_{u_0}f_{v_0(0,t_k]}=f_{v_0(0,t_k]}e_{u_0},\]
where $f_{v_0(0,t_k]}=j_{0,−1}j_{0,−2}\ldots j_{0,−t_k}$. There is a standard basis vector $\xi_0$ in
$\H$ such that $S_{u_0}\ol{\xi_0}=\ol{\xi_0}$ and, the red path of length $k$ into $\ol{\xi_0}$ is 
$f_{v_0(0,t_k]}$. The cyclic coinvariant subspace $\V$ is the minimal coinvariant subspace for $(S,T)$ containing $\xi_0$. The type 2b representations are determined similarly.

\subsection*{\texorpdfstring{Type 3}{}} Finally, for type 3 representations, both $S$ and $T$ are of infinite tail type. Type 3 representations break into 3 further sub-types: type 3a
(or inductive limit representations), type 3b(i) and type 3b(ii). In type 3a representations  it is not hard to see that every standard basis vector is wandering. We will not dwell on this type further.

For type 3b(i) case the cyclic coinvariant subspace $\V$ is described by a vector $\xi_0$ and words $u_{-1}$ of length $k$ and $v_{1}$ of length $l$ such that
	\[ S_{u_{-1}}T_{v_{1}}\ol{\xi_0}=\ol{\xi_0}. \]
We let $\ol{\xi_{-1}}=T_{v_{1}}\ol{\xi_0}$ and define $u_{-2}$ and $v_{2}$ by $S_{u_{-2}}T_{v_{2}}=T_{v_1}S_{u_{-1}}$ so that
	\[ S_{u_{-2}}T_{v_{2}}\ol{\xi_{-1}}=\ol{\xi_{-1}}. \]
Repeating this process we define basis vectors $\xi_{-n}$ and words $u_{-n}$ of length $k$ and $v_n$ of length $l$ so that
	\[ S_{u_{-n}}T_{v_{n}}\ol{\xi_{-n}}=\ol{\xi_{-n}}. \]
Similarly we define words $u_0$ and $v_{0}$ by $T_{v_{-1}}S_{u_1}=S_{u_{-1}}T_{v_{1}}$ and define $\xi_1$ by
$\ol{\xi_1}=S_{u_1}\ol{\xi_0}$. We continue in this manner to find  $\xi_{n}$ and words $u_{n}$ of length $k$ and $v_{-n}$ of length $l$ so that
	\[ T_{v_{-(n-1)}}S_{u_{(n-1)}}\ol{\xi_{n}}=\ol{\xi_{n}}. \]
Then $\V$ is the coinvariant space containing $\{\xi_n\}_{-\infty}^\infty$.

For the type 3b(ii) case, $\V$ is determined by an two infinite tails $\tau_e=u_0u_{-1}u_{-2}\ldots$ and
$\tau_f=v_0v_{-1}v_{-2}\ldots$, where $u_d\in\bF_m^+$, $|u_d|=k$ and $v_d\in\bF_n^+$, $|v_d|=l$ satisfying 
\[ f_{v_{d+1}}e_{u_d}=e_{u_{d+1}}f_{v_d}. \]
There is a basis vector $\xi_{0}$, which satisfies, \[T_{v_0v_{-1}v_{-2}\ldots v_{-t}}^*\ol{\xi_0}=S_{u_0u_{-1}u_{-2}\ldots u_{-t}}^*\ol{\xi_0},\]
for $t\ge 0$. The cyclic coinvariant subspace $\V$ is the minimal coinvariant subspace for $(S,T)$ containing $\xi_0$. 
\\[0.25in] 

\hypersetup{bookmarksdepth}
\noindent We fix a nonself-adjoint $2$-graph algebra $\fS$ generated by an irreducible Cuntz-type atomic representation $(S,T)$ of $\gr$ on $\H$. We denote by $\V$ is corresponding cyclic coinvariant subspace as described above. We wish to determine when $\fS$ has a wandering vector. Note that if a vector $\eta$ in the standard basis is not a wandering vector then one of the following must be satisfied:
\begin{align*}
 \tag{W1}\label{w1}\<S_u\eta,\eta\>&\neq0\text{ for some }u\in\bF_m^+,\ u\neq\mt,\\
 \tag{W2}\label{w2}\<T_v\eta,\eta\>&\neq0\text{ for some }v\in\bF_n^+,\ v\neq\mt,\\
 \tag{W3}\label{w3}\<T_vS_u\eta,\eta\>&\neq0\text{ for some }u\in\bF_m^+,\ v\in\bF_n^+, u,v\neq\mt,\\
\intertext{or}
 \tag{W4}\label{w4}\<S_u\eta,T_v\eta\>&\neq0\text{ for some }u\in\bF_m^+,\ v\in\bF_n^+,\ u,v\neq\mt.
\end{align*}
These conditions can be restated as
\begin{align*}
 \tag{W1$'$}\label{w1'}S_u\ol\eta&=\ol\eta\text{ for some }u\in\bF_m^+,\ u\neq\mt,\\
 \tag{W2$'$}\label{w2'}T_v\ol\eta&=\ol\eta\text{ for some }v\in\bF_n^+,\ u\neq\mt,\\
 \tag{W3$'$}\label{w3'}T_vS_u\ol\eta&=\ol\eta\text{ for some }u\in\bF_m^+,\ v\in\bF_n^+,\ u,v\neq\mt,\\
\intertext{and}
 \tag{W4$'$}\label{w4'}S_u\ol\eta&= T_v\ol\eta\text{ for some }u\in\bF_m^+,\ v\in\bF_n^+,\ u,v\neq\mt,
\end{align*}
respectively.

In the following lemma we note that not all $4$ conditions above can be satisfied in each type of atomic representations.

\begin{lem}\label{lem: wandering conditions}
	Let $(S,T)$ be an atomic representation and let $\eta$ be a standard basis vector. Then $\eta$
	satisfies condition (W$I$), for $I=1,2,3,4$, if and only if (W$I$) is satisfied by a standard
	basis vector in $\V$.

	That is, condition (W1) can only happen in type 1 and type 2a representations; condition (W2)
	can only happen in type 1 and type 2b representations; condition (W3) can only happen in type 1 and type 3b(i) representations; and condition (W4) can only happen in type 1 and type
	3b(ii) representations.
\end{lem}

\begin{proof}
	Let $(S,T)$ be an atomic representation of $\gr$ on $\H$, with cyclic coinvariant subspace $\V$. 
	Take any $p,q>0$, from \cite[Theorem 3.12]{Ful} and \cite[Lemma 5.6]{DPYdiln}, it follows that 
	\begin{equation*}
		\H=\bigvee_{\substack{n\geq0\\ d(w)=(np,nq)}}(ST)_w\V,
	\end{equation*}
	and for any standard basis vector $\zeta$ there is an  $n\geq0$, $w\in\gr$ with $d(w)=(np,nq)$, 	and a standard basis vector $\xi$ in $\V$ such that
	\begin{align}
	\label{E:xizeta} 
	(ST)_w\ol{\xi}=\ol{\zeta}. 
	\end{align}

	Suppose that there is a vector $\zeta$ in $\V^\perp$ satisfying (\ref{w1}). Hence, there is a 
	$u$ of length $p$ such that $S_u\ol{\zeta}=\ol{\zeta}$. 
	Taking $q=1$ in the above gives (\ref{E:xizeta}). 
	As $u^n$ ($u$ concatenated with itself $n$ times) is the unique path of length $np$ into $\zeta$
	and $S_{u^n}\ol{\zeta}=\ol{\zeta}$ it follows that there is a red path $v$ of length $n$ such 		that \[ T_v\ol{\xi}=\ol{\zeta}. \]
	It follows that $S_uT_v\ol{\xi}=\ol{\zeta}$.
	By the commutation relations there is a blue path $u'$ of length $p$ and a red path $v'$ of
	length $n$ such that $S_uT_v\ol{\xi}=T_{v'}S_{u'}\ol{\xi}$. By the uniqueness of the red path of
	length $n$ into $\xi$, it follows that $v'=v$ and \[S_{u'}\ol{\xi}=\ol{\xi}.\]
	We have shown that there is a standard basis vector in $\V$ satisfying (\ref{w1}). Condition 		(\ref{w2}) is dealt with similarly.

	Now, suppose that there is a standard basis vector $\zeta$ in $\V^\perp$ satisfying (\ref{w4}). 	Then there exists a blue word $u$ of length $p$ and red path  $v$ of length $q$ such that
	\[ S_u\ol{\zeta}=T_v\ol{\zeta}. \]
	Let $\omega$ be the standard basis vector such that $\ol\omega=T_v\ol{\zeta}$.
	Now, by the argument above, there is a standard basis vector $\eta$ in $\V$, an $n>0$, a blue
	word $r$ of length $np$ and a red word $s$ of length $nq$ such that $S_rT_s\ol{\eta}=\ol{\zeta}$.
	Let $\zeta'$ be the standard basis vector found by pulling back from $\zeta$ by a blue path of
	length $p$. That is, there is a $u'$ of length $p$ such that
	\[ 
	S_{u'}\ol{\zeta'} =\ol\zeta. 
	\]
	The commutation relations tell us that
	\begin{equation*}
		\ol{\omega}=T_vS_{u'}\ol{\zeta'}=S_{u''}T_{v'}\ol{\zeta'}.
	\end{equation*}
	By the uniqueness of the blue path into $\omega$ it follows that $u''=u$ and that
	$T_{v'}\ol{\zeta'}=\ol\zeta$. Hence $\zeta'$ also satisfies (\ref{w4}). By construction there is 
	blue word $r'$ of length $(n-1)p$ so that $r=u'r'$. Now  $S_{r'}T_s\ol{\eta}=\ol{\zeta'}$. 
	Continuing this process of pulling back from a standard basis vector satisfying (\ref{w4}), we
	find that $\eta$ must satisfy (\ref{w4}).

	Finally, suppose that there is a standard basis vector $\zeta_0$ satisfying (\ref{w3}).
	Hence there are words $u_{-1}$ of length $k$ and $v_{1}$ of length $l$ such that
	\[ S_{u_{-1}}T_{v_{1}}\ol{\zeta_0}=\ol{\zeta_0}. \]
	We let $\ol{\zeta_{-1}}=T_{v_{1}}\ol{\zeta_0}$ and define $u_{-2}$ and $v_{2}$ by
	$S_{u_{-2}}T_{v_{2}}=T_{v_1}S_{u_{-1}}$ so that
	\[ S_{u_{-2}}T_{v_{2}}\ol{\zeta_{-1}}=\ol{\zeta_{-1}}. \]
	Repeating this process we define basis vectors $\xi_{-n}$ and words $u_{-n}$ of length $k$ and 		$v_n$ of length $l$ so that
	\[ S_{u_{-n}}T_{v_{n}}\ol{\zeta_{-(n-1)}}=\ol{\zeta_{-(n-1)}}. \]
	Similarly we define words $u_0$ and $v_{0}$ by $T_{v_0}S_{u_0}=S_{u_{-1}}T_{v_{1}}$ and 		define $\zeta_1$ by
	$\ol{\zeta_1}=S_{u_0}\ol{\zeta_0}$. We continue in this manner to find  $\zeta_{n}$ and words
	$u_{n}$ of length $k$ and $v_{-n}$ of length $l$ so that
	\[ T_{v_{-n}}S_{u_{n}}\ol{\zeta_{n}}=\ol{\zeta_{n}}. \]
	If the sequence $\{\zeta_{n}\}$ has no repetition, then
	 $S$ is of infinite tail type and $T$ is of infinite tail type. The above 
	relations show that $(S,T)$ is of type 3b(i).

	Finally, if $\{\zeta_{n}\}$ has repetition then eventually there a finite set of vectors 
	which repeat. The coinvariant space containing these vectors defines a type 1 representation.
\end{proof}

The above lemma, in particular, shows that (\ref{w4}) can not happen in types 2 and 3b(i). This will make it easy to show that we always have wandering vectors for types 2 and 3b(i). 

\begin{prop}
	Let $(S,T)$ be an atomic representation of type 2 or type 3b(i) of a $2$-graph $\gr$ on a Hilbert
	space $\H$. Then $(S,T)$ has wandering vectors.
\end{prop}

\begin{proof}
	Suppose that $(S,T)$ is of type 2a. Let $\V$ be the corresponding cyclic coinvariant subspace. 		Choose a standard basis vector $\xi$ in $\V$. Then there is a $u$ of length $k$ such that
	\[ S_u\ol\xi=\ol\xi. \]
	Let $u'$ be a blue word of length $k$, such that $u'\neq u$. Then $S_{u'}\xi$ is in $\V^\perp$. 
	Clearly, $S_{u'}\xi$ does not satisfy (\ref{w2}) or (\ref{w3}). 
	As $S_{u'}\xi$ can not satisfy (\ref{w1}), it follows from Lemma \ref{lem: wandering conditions}
	that $S_{u'}\xi$ is a wandering vector for the representation $(S,T)$. Type 2b is dealt with
	similarly.

	Now suppose that $(S,T)$ is of type 3b(i), with corresponding cyclic coinvariant subspace
	$\V$. If we choose a standard basis vector $\xi$ in $\V$, we can find a blue word $u$ of length 	$k$ and a red word $v$ of length $l$ so that \[T_vS_u\ol\xi=\ol\xi.\]
	Again, if we take a blue word $u'$ of length $k$ where $u'\neq u$, then $S_{u'}\xi$ lies in 
	$\V^\perp$ and can not satisfy (\ref{w3}). 
	It is also easy to see that $S_{u'}\xi$ does not satisfy (\ref{w1}) or (\ref{w2}).
	Hence $S_{u'}\xi$ is a wandering vector for the representation $(S,T)$.
\end{proof}

Condition (\ref{w4}) has proven to be the most difficult to deal with. While the above argument shows that if $\zeta$ satisfies (\ref{w4}) we can pull back to a vector $\zeta'$ which also satisfies (\ref{w4}) with words of the same length, it tells us nothing about going forward. We will see that periodicity of $\gr$ can play a role here. 

\begin{example}\label{ex: periodic}
	Let $A=[A_1,A_2]$ be the $1$-dimensional representation of $\bF_2^+$ on $\spn\{\xi_\mt\}$ given by
	\begin{equation*}
		A_1\xi_\mt=\xi_\mt\text{ and } A_2\xi_\mt=0.
	\end{equation*} 
	Now let $S$ be the minimal isometric dilation of $A$. Since $A$ is an atomic representation of
	$\bF_2^+$, $S$ will be an atomic representation. If we let $T=S$ then, by
	Proposition \ref{prop: fsg=2 graph alg}, $(S,T)$ forms a representation of a periodic $2$-graph.
	Whilst $S$ has many wandering vectors as a representation of $\bF_2^+$ (see \cite{DKS}), since
	$S=T$ we can not have any wandering vectors in the $2$-graph sense. 
\end{example}

This example is typical of what happens in every type 1 or type 3b(ii) representation when $\gr$ is periodic. Indeed, we will conclude that there are no wandering vectors when $\gr$ is periodic by showing that the row-isometries $[S_u: |u|=a]$ and $[T_v: |v|=b]$ are equal for some $a,b>0$. 

\begin{prop}\label{prop: periodic no wandering}
	Let $(S,T)$ be an irreducible type 1 or type 3b(ii) atomic representation of a periodic $2$-graph $\gr$. Then $(S,T)$ has no wandering vectors.
\end{prop}

\begin{proof}
	Suppose that $\gr$ has $(a,-b)$-periodicity. Let $\gamma$ be the bijection defined in Theorem 		\ref{thm: periodic cond}. Then it also has $(pa,-pb)$-periodicity for any
	non-negative integer $p$.
	Recall that in the type 1 case we have that $\V$ is determined by words $u_0$ and $v_0$ with
	$|u_0|=k$ and $|v_0|=l$ and a basis vector $\xi$ such that
	\[ S_{u_0}\ol\xi=T_{v_0}\ol\xi. \]
	In the type 3b(ii) case $\V$ is determined by $\tau_e=u_0u_{-1}u_{-2}\ldots$ and
	$\tau_f=v_0v_{-1}v_{-2}\ldots$, where $u_d\in\bF_m^+$, $|u_d|=k$ and $v_d\in\bF_n^+$, $|v_d|=l$ 	satisfying 
	\[ f_{v_{d+1}}e_{u_d}=e_{u_{d+1}}f_{v_d}. \]
	Where $\V$ is given by pulling back from a basis vector $\xi$ by $\tau_e$ and $\tau_f$. In both
	cases, by replacing $k$, $l$, $a$ and $b$ by suitable multiples can assume that $k=a$ and $l=b$.

	Suppose $\eta$ is a basis vector in $\H$ such that
	\begin{equation*}
	S_u\ol\eta=T_v\ol\eta,
	\end{equation*}
	with $|u|=a$, $|v|=b$. Hence $\eta$ satisfies (\ref{w4}) and is not wandering. Our first goal is 	 to show that $\gamma(u)=v$. Since we have a Cuntz-type representation there is a basis vector
	$\eta'$ and blue path $u'$ of length $a$ such that $S_{u'}\ol{\eta'}=\ol\eta$. By the commutation
	relations we necessarily have that the red bath $v'$ leading into $\eta$ also comes from
	$\eta'$. That is, we have \[ S_{u'}\ol{\eta'}=T_{v'}\ol{\eta'}=\ol{\eta}. \]
	By the commutation relations and Theorem \ref{thm: periodic cond}
	\begin{align*}
		S_uT_{v'}\ol{\eta'}&=T_vS_{u'}\ol{\eta'}\AND \\
		S_uT_{v'}\ol{\eta'}&=T_{\gamma(u)}S_{\gamma^{-1}(v')}\ol\eta,
	\end{align*}
	and hence $v=\gamma(u)$.

	Now let
	$\zeta={S_u\eta}$. We will show that $\zeta$ also satisfies (\ref{w4}). Choose any
	$u'\in\bF_m^+$ with $|u'|=a$. Then by the commutation relations we have
	\begin{align*}
		S_{u'}T_{\gamma(u)}\ol\eta=T_{\gamma(u')}S_u\ol\eta.
	\end{align*}
	Hence we have $S_{u'}\ol\zeta=T_{\gamma(u')}\ol\zeta$.

	Take any basis vector $\zeta$ in $\V$. By construction there is a vector $\eta\in\V$ and
	$|u|=a$, $|v|=b$ such that 
	\[ S_u\ol\eta=T_v\ol\eta=\ol\zeta. \]
	Hence the above argument applies to $\zeta$.
	It follows that the row-isometries $[S_u: |u|=a]$ and $[T_{\gamma(u)}: |u|=a]$ are equal.

	By \cite[Theorem 3.12]{Ful}
	\[ \H=\bigvee_{d(w)=(a,b)}(ST)_w\V. \]
	Hence, since $[S_u: |u|=a]=[T_{\gamma(u)}: |u|=a]$, we also have that
	\[ \H=\bigvee_{|u|=a}S_a\V=\bigvee_{|u|=a}T_{\gamma(u)}\V. \]
	Hence if we have any basis vector $\eta$ in $\V^\perp$ then $\eta$ must satisfy (\ref{w4}), and 	hence $\eta$ is not wandering. 
\end{proof}

Whether the type 1 and type 3b(ii) representations of aperiodic $2$-graph algebras have standard basis vectors as wandering vectors remains open. We can, however, give a sufficient condition for wandering vectors to exist for type 1 representations.

\begin{prop}\label{prop: fin cor wand}
Let $(S,T)$ be an irreducible atomic representation of type 1 of $\gr$ with minimal finite dimensional coinvariant cyclic subspace $\V$. 
If there is a standard basis vector in $\V^\perp$ satisfying (\ref{w1}) or (\ref{w2}) then there is a standard basis vector which is wandering for $(S,T)$.
\end{prop}

\begin{proof}
	Assume that $\zeta$ is a standard basis vector in $\V^\perp$ satisfying (\ref{w2}). Then there is a red
	ring $v$ of length $q$ such that $T_v\ol\zeta=\ol\zeta$. Let $v'$ be another path of length
	$|v|$ 	with
	$v'\neq v$ and let $\zeta'$ be the standard basis vector given by $T_{v'}\ol\zeta=\ol{\zeta'}$. We
	claim that $\zeta'$ is wandering.

	Clearly, by the uniqueness of red paths into $\zeta'$, $\zeta'$ cannot satisfy condition (\ref{w2}).
	
	Suppose that $\zeta'$ satisfies (\ref{w1}). Then there is a blue path $u'$ of length $q$ so that
	$S_{u'}\ol{\zeta'}=\ol{\zeta'}$. By the commutation relations and the uniqueness of red paths into $\zeta'$, we have
	\begin{align*}
		\zeta'&=S_{u'}\ol{\zeta'}= S_{u'}T_{v'}\ol{\zeta}\\
		&= T_{v'}S_u\ol\zeta=T_{v'}\ol\zeta.
	\end{align*}
	So $S_u\ol\zeta=\ol\zeta$. Thus $\zeta$ satisfies both  (\ref{w1}) and (\ref{w2}). 
	This contradicts the irreducibility of the representation $(S,T)$.  
	
	If $\zeta'$ satisfies (\ref{w3}), then there are blue and red paths $\alpha$ and $\beta$ of length $p$ and $q'$ respectively,  such that 
	$T_\beta S_\alpha\ol{\zeta'}=\ol{\zeta'}$.  Pushing forward in blue eventually reaches a basis vertex, say $\eta$, on the red ring $v$. 
	Pushing forward from $\eta$ in blue further ${\text{lcm}(q, q')}/{q'}$ times, we obtain another blue path of length $p$ (from a vertex on $v$)  leading into $\eta$. 
	This obviously yields a contradiction. 
	
	Finally let us suppose that $\zeta'$ satisfies (\ref{w4}). Then there are blue and red paths $\alpha$ and $\beta$ of length $p$ and $q'$ respectively,
	such that $S_\alpha\ol{\zeta'}=T_\beta\ol{\zeta'}$. This time, pulling back in blue eventually reaches a basis vertex, say $\eta$, on the red ring $v$. 
	Then we obtain a blue ring at $\eta$ by pulling back from $\eta$ in blue $\text{lcm}(q, q')/q'$ times. Hence $\eta$ satisfies both (\ref{w1}) and 
	(\ref{w2}), which yields a contradiction as above.	
	
	The case that $\zeta$ satisfies (\ref{w1}) is argued similarly. 
	\end{proof}

The following example illustrates Proposition \ref{prop: fin cor wand}.

\begin{example}\label{ex: aperiodic}
	Let $A=[A_1,A_2]$ and $B=[B_1,B_2]$ be a $1$-dimensional representation of $\bF_{\mathrm{id}}^+$ 		on $\spn\{\xi_\mt\}$, where $\xi_\mt$ is a unit vector and $\mathrm{id}$ is the identity
	representation in $S_{2\times 2}$, such that
	\begin{equation*}
		A_1\xi_\mt=\xi_\mt=B_1\xi_\mt.
	\end{equation*}
	This defines an atomic representation. Let $(S,T)$ be the minimal isometric representation of 		$(A,B)$. Then $(S,T)$ is Cuntz-type, atomic representation \cite{DPYdiln}.
	Let $\fS$ be the nonself-adjoint $2$-graph algebra generated by $(S,T)$. Let
	$\zeta=S_2T_2\xi_\mt=T_2S_2\xi_\mt$. Note that $S_1T_2\xi_\mt=T_2\xi_\mt$, thus $T_2\xi_\mt$ 		satisfies (\ref{w1}). By Proposition \ref{prop: fin cor wand} $\zeta$ is a wandering vector. Let 		$\M=\fS[\zeta]$. Then
	$\fS|_\M\cong\L_{\mathrm{id}}$. Thus
	\begin{equation*}
		\fS=\begin{bmatrix}
			\bC I&0&0\\
			*&*&0\\
			*&*&\L_{\mathrm{id}}
		\end{bmatrix}.
	\end{equation*}
\end{example}

\hypersetup{bookmarksdepth=-3}\section*{}
\subsection*{Acknowledgements} The authors would like to thank Ryan Hamilton for the many fruitful conversations and ideas he shared. 
They also thank Prof. Ken Davidson and Prof. Matt Kennedy for their useful comments. 
\hypersetup{bookmarksdepth}


\end{document}